\documentclass[12pt,reqno]{amsart}

\usepackage{multirow}
\usepackage{hhline}

\usepackage{mathtools}
\usepackage{times}
\usepackage[T1]{fontenc}
\usepackage{mathrsfs}
\usepackage{latexsym}
\usepackage[dvips]{graphics}
\usepackage[titletoc, title]{appendix}
\usepackage{amsmath,amsfonts,amsthm,amssymb,amscd}
\usepackage[dvipsnames]{xcolor}
\usepackage{hyperref}
\usepackage{amsmath}
\usepackage[utf8]{inputenc}
\usepackage{tikz-cd}
\usepackage{color}
\usepackage{breakurl}

\usepackage{comment}
\newcommand{\bburl}[1]{\textcolor{blue}{\url{#1}}}

\newtheorem{thm}{Theorem}[section]

\newtheorem{cor}[thm]{Corollary}

\newtheorem{lem}[thm]{Lemma}
\newtheorem{prop}[thm]{Proposition}

\usepackage[utf8]{inputenc}

\DeclareFixedFont{\ttb}{T1}{txtt}{bx}{n}{12} 
\DeclareFixedFont{\ttm}{T1}{txtt}{m}{n}{12}  

\usepackage{color}
\definecolor{deepblue}{rgb}{0,0,0.5}
\definecolor{deepred}{rgb}{0.6,0,0}
\definecolor{deepgreen}{rgb}{0,0.5,0}

\usepackage{listings}

\newcommand\pythonstyle{\lstset{
language=Python,
basicstyle=\ttm,
morekeywords={self},              
keywordstyle=\ttb\color{deepblue},
emph={MyClass,__init__},          
emphstyle=\ttb\color{deepred},    
stringstyle=\color{deepgreen},
frame=tb,                         
showstringspaces=false
}}

\lstnewenvironment{python}[1][]
{
\pythonstyle
\lstset{#1}
}
{}

\newcommand\pythoninline[1]{{\pythonstyle\lstinline!#1!}}

\numberwithin{equation}{section}

\DeclareFontFamily{U}{mathx}{}
\DeclareFontShape{U}{mathx}{m}{n}{<-> mathx10}{}
\DeclareSymbolFont{mathx}{U}{mathx}{m}{n}
\DeclareMathAccent{\widehat}{0}{mathx}{"70}
\DeclareMathAccent{\widecheck}{0}{mathx}{"71}

\begin{document}


\title{A Pair of Diophantine Equations\\ Involving the Fibonacci Numbers}

\author[Chen]{Xuyuan Chen}
\email{\textcolor{blue}{\href{mailto:xuyuanc@andrew.cmu.edu}{xuyuanc@andrew.cmu.edu}}}
\address{Mathematical Sciences\\ Carnegie Mellon University\\ Pittsburgh, PA 15213, USA}

\author[Chu]{H\`ung Vi\d{\^e}t Chu}
\email{\textcolor{blue}{\href{mailto:hungchu1@tamu.edu}{hungchu1@tamu.edu}}}
\address{Department of Mathematics\\ Texas A\&M University, College Station, TX 77843, USA}

\author[Kesumajana]{Fadhlannafis K. Kesumajana}
\email{\textcolor{blue}{\href{mailto:nafis.khowarizmi@gmail.com}{nafis.khowarizmi@gmail.com}}}\address{Department of Mathematics\\ Institut Teknologi Bandung, Bandung, Jawa Barat 40132, Indonesia}

\author[Kim]{Dongho Kim}
\email{\textcolor{blue}{\href{mailto:dtkim25@berkeley.edu}{dtkim25@berkeley.edu}}}
\address{Department of Mathematics\\ University of California, Berkeley, Berkeley, CA 94720, USA}

\author[Li]{Liran Li}
\email{\textcolor{blue}{\href{mailto:zqj6pe@virginia.edu}{zqj6pe@virginia.edu}}}
\address{Department of Mathematics\\ University of Virginia, Charlottesville, VA 22904, USA}

\author[Miller]{Steven J. Miller}
\email{\textcolor{blue}{\href{mailto:sjm1@williams.edu}{sjm1@williams.edu},
\href{mailto:Steven.Miller.MC.96@aya.yale.edu}{Steven.Miller.MC.96@aya.yale.edu}}}
\address{Department of Mathematics and Statistics, Williams College, Williamstown, MA 01267, USA}

\author[Yang]{Junchi Yang}
\email{\textcolor{blue}{\href{mailto:j647yang@uwaterloo.ca}{j647yang@uwaterloo.ca}}}
\address{Department of Pure Mathematics\\ University of Waterloo, Ontario, N2L 3G1, Canada}

\author[Chris]{Chris Yao}
\email{\textcolor{blue}{\href{mailto:chris.yao@berkeley.edu}{chris.yao@berkeley.edu}}}
\address{Department of Mathematics\\ University of California, Berkeley, Berkeley, CA 94720, USA}



\thanks{We thank the participants at Polymath Jr. 2024 REU for helpful discussions.}

\maketitle

\begin{abstract}
Let $a, b\in \mathbb{N}$ be relatively prime. Previous work showed that exactly one of the two equations $ax + by = (a-1)(b-1)/2$ and $ax + by + 1 = (a-1)(b-1)/2$ has a nonnegative, integral solution; furthermore, the solution is unique. Let $F_n$ be the $n$\textsuperscript{th} Fibonacci number. When $(a,b) = (F_n, F_{n+1})$, it is known that there is an explicit formula for the unique solution $(x,y)$. We establish formulas to compute the solution when $(a,b) = (F_n^2, F_{n+1}^2)$ and $(F_n^3, F_{n+1}^3)$, giving rise to some intriguing identities involving Fibonacci numbers. Additionally, we construct a different pair of equations that admits a unique positive (instead of nonnegative), integral solution.
\end{abstract}

\section{Introduction}

The study of Diophantine equations, particularly those associated with specific integer sequences, has been a topic of great interest in number theory (see \cite{KCC, GGL1, GGL2, GS, I, Tac} for some recent work on Diophantine equations that involve Fibonacci numbers). 
We study a pair of Diophantine equations, first encountered by Beiter \cite{Be} in the study of cyclotomic polynomials $\Phi_{pq}(x)$ for primes $p < q$. The pair was later extended by Chu \cite{C4} to relatively prime numbers. 

\begin{thm}\cite[Theorem 1.1]{C4}\label{o1}
For relatively prime $a, b\in \mathbb{N}$, exactly one of the following equations has a nonnegative, integral solution
\begin{align}
\label{eq1}ax + by &\ =\ \frac{(a-1)(b-1)}{2},\\
\label{eq2} ax + by + 1 &\ =\ \frac{(a-1)(b-1)}{2}.
\end{align}
Furthermore, the solution is unique. 
\end{thm}

Thanks to Theorem \ref{o1}, we can define the function $\Gamma: \{(m,n): \gcd(m, n) = 1\}\rightarrow \{1,2\}$ as
$$\Gamma(m,n) \ =\ \begin{cases}1,&\mbox{ if }\eqref{eq1} \mbox{ has a solution when }(a,b) = (m,n);\\2,&\mbox{ if }\eqref{eq2} \mbox{ has a solution when }(a,b) = (m,n).\end{cases}$$ It is well-known that two consecutive Fibonacci numbers are relatively prime. Chu \cite{C4} then studied $\Gamma(F_n, F_{n+1})$, where $(F_n)_{n=0}^\infty$ are the Fibonacci numbers defined as $F_0 = 0, F_1 = 1$, and $F_n = F_{n-1} + F_{n-2}$ for $n\ge 2$. It turned out that for $n\ge 3$, $\Gamma(F_n, F_{n+1})$ alternates between $1$ and $2$ in groups of three (\cite[Theorem 1.6]{C4}). Continuing the work, Davala \cite{Da} examined  $\Gamma(B_n, B_{n+1})$, $\Gamma(B_{2n-1}, B_{2n+1})$, $\Gamma(C_n, C_{n+1})$, and $\Gamma(B_n, C_n)$,  with $(B_n)_{n=1}^\infty$ and $(C_n)_{n=1}^\infty$ being the so-called $n$\textsuperscript{th} balancing number and the $n$\textsuperscript{th} Lucas-balancing number.\footnote{Balancing numbers were introduced by Behera and Panda \cite{BB} to be solutions of the Diophantine equation $1 + 2 + \cdots + (n - 1) = (n + 1) + (n + 2) + \cdots + (n + r)$ for some natural number $r$. The $n$\textsuperscript{th} balancing number is denoted by $B_n$, and $C_n = \sqrt{8B_n^2+1}$ is called the $n$\textsuperscript{th} Lucas number \cite{R}.} These sequences  satisfy the recurrence relations $B_1 = 1, B_2 = 6, B_{n+1} = 6B_n - B_{n - 1}$ and $C_1 = 3, C_2 = 17, C_{n+1} = 6C_n - C_{n-1}$ for $n \ge 2$.

Recently, Arachchi et al.\ \cite{ACLLMM} provided a useful criterion on $(a,b)$ to determine the value of $\Gamma(a,b)$.
The authors then used the criterion to study $\Gamma(a_n, a_{n+1})$ for various sequences $(a_n)_{n=1}^\infty$, including the natural numbers raised to the $k$\textsuperscript{th} power, arithmetic progressions, shifted geometric sequences, and so on. One notable result is that for a fixed $k$, $\Gamma(n^k, (n+1)^k)$ eventually alternates between $1$ and $2$. 

Our main results are inspired by the identities in \cite[Theorem 1.6]{C4}:  
\begin{align*}
\frac{F_{6k-1}-1}{2}F_{6k}+\frac{F_{6k-1}-1}{2}F_{6k+1}&\ =\ \frac{(F_{6k}-1)(F_{6k+1}-1)}{2},\\
\frac{F_{6k+1}-1}{2}F_{6k+1}+\frac{F_{6k-1}-1}{2}F_{6k+2}&\ =\ \frac{(F_{6k+1}-1)(F_{6k+2}-1)}{2},\\
\frac{F_{6k+1}-1}{2}F_{6k+2}+\frac{F_{6k+1}-1}{2}F_{6k+3}&\ =\ \frac{(F_{6k+2}-1)(F_{6k+3}-1)}{2},\\
1+\frac{F_{6k+2}-1}{2}F_{6k+3}+\frac{F_{6k+2}-1}{2}F_{6k+4}&\ =\ \frac{(F_{6k+3}-1)(F_{6k+4}-1)}{2},\\
1+\frac{F_{6k+4}-1}{2}F_{6k+4}+\frac{F_{6k+2}-1}{2}F_{6k+5}&\ =\ \frac{(F_{6k+4}-1)(F_{6k+5}-1)}{2},\\
1+\frac{F_{6k+4}-1}{2}F_{6k+5}+\frac{F_{6k+4}-1}{2}F_{6k+6}&\ =\ \frac{(F_{6k+5}-1)(F_{6k+6}-1)}{2}.
\end{align*}
Therefore, not only is $\Gamma(F_n, F_{n+1})$ periodic, but there is also a formula for the solution $(x,y)$. We shall examine $\Gamma(F^2_n, F^2_{n+1})$ and establish a formula for the solution $(x,y)$. Let us look at the data.

\begin{center}
\begin{tabular}{ |c|c|c|c|c|c| } 
 \hline
 $n$ & $F_n^2$ & $F_{n+1}^2$ & $x_n$ & $y_n$ & $\Gamma(F_n^2, F_{n+1}^2)$ \\
 \hline
2 & 1 & 4 & 0 & 0 & 1\\
3 & 4 & 9 & 3 & 0 & 1 \\ 
4 & 9 & 25 & 5 & 2 & 2 \\
5 & 25 & 64 & 20 & 4 & 1 \\
6 & 64 & 169 & 51 & 12 & 1 \\
7 & 169 & 441 & 83 & 52 & 2\\
8 & 441 & 1156 & 356 & 84 & 1\\
9 & 1156 & 3025 & 935 & 220 & 1\\
10 & 3025 & 7921 & 1513 & 934 &2\\
11 & 7921 & 20736 & 6408 & 1512 & 1\\
12 & 20736 & 54289 & 16775 & 3960 & 1\\
13 & 54289 & 142129 & 27143 & 16776 &2\\
 \hline
\end{tabular}
\begin{center}Table 1. Data for $\Gamma(F_n^2, F_{n+1}^2)$.\end{center}
\end{center}

The data suggest that as $n$ increases, $\Gamma(F_n^2, F_{n+1}^2)$ repeats the pattern $1, 1, 2$. Furthermore, the unique nonnegative, integral solution $(x_n, y_n)$ of the Diophantine pair 
\begin{align}
    \label{eq6}F_n^2x_n + F_{n+1}^2y_n &\ =\ \frac{(F_n^2-1)(F_{n+1}^2-1)}{2},\\
    \label{eq7}1+ F_n^2x_n + F_{n+1}^2y_n &\ =\ \frac{(F_n^2-1)(F_{n+1}^2-1)}{2}
\end{align}
seems to have the property that for $n\not \equiv 1\mod 3$, 
$$F_n^2 - x- y\ =\ 1,$$
while 
$$\begin{cases}
2x_n - F_n^2 \ =\ 1,&\mbox{ for }n \equiv 4\mod 6;\\
2x_n - F_n^2\ =\ -3, &\mbox{ for }n \equiv 1\mod 6.
\end{cases}$$

The next theorem and proposition confirm that our observations hold for all $n\ge 2$:

\begin{thm}\label{m1}
For $n\ge 2$, we have the following identities:
\begin{align}
&\label{eq3}1 + \frac{F_n^2 - 3}{2} \cdot F_n^2 + \frac{F_n^2 - F_{n-1}^2 - 1}{2} \cdot F_{n+1}^2 \ = \ \frac{(F_{n}^2 - 1) ( F_{n+1}^2 - 1)}{2},  {\rm if \ } n{\rm  \ is\ odd};\\
&\label{eq4}1 + \frac{F_n^2 + 1}{2} \cdot F_n^2 + \frac{F_n^2 - F_{n-1}^2 - 1}{2} \cdot F_{n+1}^2 \ = \ \frac{(F_{n}^2 - 1) ( F_{n+1}^2 - 1)}{2}, {\rm if \ } n{\rm  \ is\ even};\\
&\label{eq5}\left(F_n^2 - \frac{F_{n-1}^2 +1}{2} \right) \cdot F_n^2 + \frac{F_{n-1}^2 - 1}{2} \cdot F_{n+1}^2\ = \ \frac{(F_n^2 - 1) (F_{n+1}^2 - 1)}{2}.
\end{align}
\end{thm}

\begin{prop}\label{p1}
For $n \ge 0$ with $n \equiv 0 \mod 3$, $F_n$ is even, while $F_{n+1}$ and $F_{n+2}$ are odd.
\end{prop}

\begin{cor}
For $n\ge 2$ and $n\equiv 0, 2, 3, 5\mod 6$, 
$$(x, y)\ =\ \left(F_n^2 - \frac{F_{n-1}^2 +1}{2}, \frac{F_{n-1}^2 - 1}{2}\right)$$
is the unique nonnegative, integral solution of \eqref{eq6}.

For $n\ge 2$ and $n\equiv 1\mod 6$, 
$$(x,y) \ =\ \left(\frac{F_n^2-3}{2}, \frac{F_n^2-F_{n-1}^2-1}{2}\right)$$
is the unique nonnegative, integral solution of \eqref{eq7}.

For $n\ge 2$ and $n\equiv 4\mod 6$, 
$$(x,y) \ =\ \left(\frac{F_n^2+1}{2}, \frac{F_n^2-F_{n-1}^2-1}{2}\right)$$
is the unique nonnegative, integral solution of \eqref{eq7}.
\end{cor}

Next, we investigate $\Gamma(F^3_n, F^3_{n+1})$ and find formulas for the unique integral solution $(x,y)$ of the pair 
\begin{align}
    \label{eq8}F_n^3x_n + F_{n+1}^3y_n &\ =\ \frac{(F_n^3-1)(F_{n+1}^3-1)}{2},\\
    \label{eq9}1+ F_n^3x_n + F_{n+1}^3y_n &\ =\ \frac{(F_n^3-1)(F_{n+1}^3-1)}{2}.
\end{align}
Interestingly, formulas for the solution can be expressed as a (alternate) sum of Fibonacci numbers cubed. As before, we examine the data.

\begin{center}
\begin{tabular}{ |c|c|c|c|c|c| } 
 \hline
 $n$ & $F_n^3$ & $F_{n+1}^3$ & $x_n$ & $y_n$ & $\Gamma(F_n^3, F_{n+1}^3)$ \\
 \hline
2 & 1 & 8 & 0 & 0 & 1\\
3 & 8 & 27 & 8 & 1 & 1\\
4 & 27 & 125 & 18 & 9 & 2\\
5 & 125 & 512 & 106 & 36 & 1\\
6 & 512 & 2197 & 405 & 161 & 2\\
7 & 2197 & 9261 & 1791 & 673 & 1\\
8 & 9261 & 39304 & 7469	& 2870 & 2\\
 \hline
\end{tabular}
\begin{center}Table 2. Data for $\Gamma(F_n^3, F_{n+1}^3)$.\end{center}
\end{center}
From the data, $\Gamma(F_n^3, F_{n+1}^3)$ seems to alternate between $1$ and $2$, and 
$$x_4 \ =\ 3^3 - x_3 - 1,\quad y_4 \ =\ y_3 + 2^3,$$ 
$$x_5 \ =\  5^3 - x_4 - 1,\quad y_5 \ =\ y_4 + 3^3,$$
$$x_6 \ =\ 8^3 - x_5 - 1, \quad y_6 \ =\ y_5 + 5^3,$$
which suggests that for $n\ge 4$,
$$\begin{cases}
        x_n &\ = \ F_n^3 - x_{n-1} -1,\\
        y_n &\ = \ y_{n-1} + F_{n-1}^3.
\end{cases}$$
These observations are verified for all $n\ge 3$ by our next theorem. 

\begin{thm}\label{m2}
    For $m \in \mathbb{N}_{\ge 2}$, we have the following identities:
    \begin{align}
        \label{eq12}\left( \sum_{k=1}^{2m - 1} (-1)^{k-1} F_k^3\right) F_{2m-1}^3 + \left( \sum_{k=2}^{2m-2} F_{k}^3 \right) F_{2m}^3 \ = \ \frac{(F_{2m-1}^3 - 1) (F_{2m}^3 - 1)}{2}, \\
        \label{eq13}1 + \left( \sum_{k=1}^{2m} (-1)^{k} F_k^3 - 1\right) F_{2m}^3 + \left( \sum_{k=2}^{2m-1} F_{k}^3 \right) F_{2m+1}^3 \ = \ \frac{(F_{2m}^3 - 1) (F_{2m+1}^3 - 1)}{2}.
    \end{align}
\end{thm}

We shall prove Theorems \ref{m1} and \ref{m2} in Sections \ref{fibsquared} and \ref{fibcubed}, respectively. In Section \ref{anotherpair}, we construct another pair of Diophantine equations that resembles the original pair, \eqref{eq1} and \eqref{eq2}, but 
admits a unique positive (instead of nonnegative), integral solution and is asymmetric in the sense that $\Gamma(a,b)$ may not be equal to $\Gamma(b,a)$. Finally, Section \ref{problems} is devoted to discussing several problems for future investigation.

\section{The case of Fibonacci numbers squared}\label{fibsquared}

The main goal of this section is to prove an analog of Theorem \ref{o1} for Fibonacci numbers squared. Our proof uses Cassini's identity and the Fibonacci recurrence relation to make similar terms appear and cancel themselves out.  

\begin{proof}[Proof of Proposition \ref{p1}]
    Let $k \ge 0$ such that $k \equiv 0 \mod 3$. Assume that $F_{k}$ is even, and $F_{k+1}$ and $F_{k+2}$ are odd. By the linear recurrence,
    \begin{align*}
        F_{k+3} \ = \ F_{k+1} + F_{k+2} \ \equiv \ 1 + 1 \ \equiv \ 0 \mod 2, \\
        F_{k+4} \ = \ F_{k+2} + F_{k+3} \ \equiv \ 1 + 0 \ \equiv \ 1 \mod 2, \\
        F_{k+5} \ = \ F_{k+3} + F_{k+4} \ \equiv \ 0 + 1 \ \equiv \ 1 \mod 2.
    \end{align*}
    Since $F_0 = 0$ and $F_1 = F_2 = 1$, Proposition \ref{p1} holds inductively.
\end{proof}

\begin{proof}[Proof of Theorem \ref{m1}]
We start by proving \eqref{eq3}. 
By Cassini's identity for odd $n$,
    $$
    F_{n}^2 - 1 \ = \ F_{n-1} F_{n+1}.
    $$
    Hence,
    $$
    F_{n}^4 - 2F_n^2 + 1 \ = \ F_{n-1}^2 F_{n+1}^2, 
    $$
    which gives
    $$
    F_{n}^4 - 2F_{n}^2 + 1 - F_{n+1}^2 \cdot F_{n-1}^2 + (1 + F_{n+1}^2 \cdot F_{n}^2 - F_{n+1}^2 - F_n^2) \ = \ (1 + F_{n+1}^2 \cdot F_{n}^2 - F_{n+1}^2 - F_n^2).
    $$
    Therefore,
    $$
    2 + ( F_n^2 - 3) \cdot F_n^2 + (F_n^2 -F_{n-1}^2 - 1) \cdot F_{n+1}^2 \ = \ (F_{n+1}^2 - 1 ) (F_{n}^2 - 1).
    $$
    We obtain
    $$
    1 + \frac{F_n^2 - 3}{2} \cdot F_n^2 + \frac{F_n^2 - F_{n-1}^2 - 1}{2} \cdot F_{n+1}^2 \ = \ \frac{(F_{n}^2 - 1) ( F_{n+1}^2 - 1)}{2}. 
    $$

Next, we prove \eqref{eq4}. We again start with Cassini's identity for even $n$:
    $$
    F_{n}^2 + 1 \ = \ F_{n-1} F_{n+1}.
    $$
    Hence,
    $$
    F_n^4 + 2F_n^2 + 1 \ = \ F_{n-1}^2 F_{n+1}^2,
    $$
    which gives
    $$
    F_n^4 + F_n^2 + 1 - F_{n-1}^2 F_{n+1}^2 + (1 + F_{n+1}^2 F_n^2 - F_{n+1}^2) \ = \ - F_n^2 + (1 + F_{n+1}^2 F_n^2 - F_{n+1}^2).
    $$
    The above implies that
    $$
    2 + ( F_n^2 + 1) \cdot F_n^2 + (F_n^2 -F_{n-1}^2 - 1) \cdot F_{n+1}^2 \ = \ (F_{n+1}^2 - 1 ) (F_{n}^2 - 1).
    $$
    Dividing both sides by $2$, we obtain 
    $$
    1 + \frac{F_n^2 + 1}{2} \cdot F_n^2 + \frac{F_n^2 - F_{n-1}^2 - 1}{2} \cdot F_{n+1}^2 \ = \ \frac{(F_{n}^2 - 1) ( F_{n+1}^2 - 1)}{2}. 
    $$

It remains to prove \eqref{eq5}. 
    On the one hand, from $F_n = F_{n+1} - F_{n-1}$, we have
    $$
        F_n^2 \ = \ F_{n+1}^2 + F_{n-1}^2 - 2 F_{n+1} F_{n-1}.
    $$
    Hence,
    $$
    F_n^2 - F_{n-1}^2 - F_{n+1}^2 \ =\ - 2 F_{n-1} F_{n+1},
    $$
    and so,
    \begin{equation}\label{eq10}
    F_n^2 (F_n^2 - F_{n-1}^2 - F_{n+1}^2) \ = \ - 2 F_n^2 F_{n-1} F_{n+1}.
    \end{equation}
    On the other hand, by Cassini's identity,
    $$
    \left( F_{n-1} F_{n+1} -  F_{n}^2 \right)^2 \ =\ 1,
    $$
    which gives
    \begin{equation}\label{eq11}
        F_n^4- 2 F_n^2 F_{n+1} F_{n-1}\ =\ 1 - F_{n-1}^2 F_{n+1}^2.
    \end{equation}
    It follows from \eqref{eq10} and \eqref{eq11} that
    \begin{align*}
        2F_n^4 - F_n^2 F_{n-1}^2 - F_n^2 F_{n+1}^2 &\ = \ F_n^4 - 2 F_n^2 F_{n-1} F_{n+1}\\
        &\ = \ 1 - F_{n-1}^2 F_{n+1}^2.
    \end{align*}
    Equivalently,
    $$
    2 F_n^4 - F_n^2 F_{n-1}^2 + F_{n-1}^2 F_{n+1}^2 \ = \ 1 + F_{n}^2 F_{n+1}^2,
    $$
    from which we know that
    $$
    2 F_n^4 - F_n^2 F_{n-1}^2 + F_{n-1}^2 F_{n+1}^2 + (- F_n^2 - F_{n+1}^2) \ = \ 1 + F_{n}^2 F_{n+1}^2 + (- F_n^2 - F_{n+1}^2),
    $$
    Therefore,
    $$
    (2F_n^2 - F_{n-1}^2 - 1) \cdot F_n^2 + (F_{n-1}^2 - 1) \cdot F_{n+1}^2 \ = \ (F_n^2 - 1) (F_{n+1}^2 - 1).
    $$
    Dividing both sides by $2$, we obtain 
    $$
    \left(F_n^2 - \frac{F_{n-1}^2 +1}{2} \right) \cdot F_n^2 + \frac{F_{n-1}^2 - 1}{2}  \cdot F_{n+1}^2 \ = \ \frac{(F_n^2 - 1) (F_{n+1}^2 - 1)}{2},
    $$
    as desired. 
\end{proof}

\section{The case of Fibonacci numbers cubed}\label{fibcubed}

This section proves an analog of Theorem \ref{o1} for Fibonacci numbers cubed. As for the case of Fibonacci numbers squared, our proof uses Cassini's identity and the Fibonacci recurrence relation. Our proofs also employ the identities presented by Frontczak \cite{F} to replace the sum of cubes with a simpler expression and use the well-known identity $F_{3n} = 5F_n^3 + 3(-1)^nF_n$ (see \cite[(2.6)]{F}).

\begin{thm} \cite[cf.\ Theorem 1]{F}\label{km}
    For each $n \geq 1$,
    \begin{equation}\label{e2}
    \sum_{k=1}^n F_{k}^3 \ = \ \frac{1}{4} \left(F_{3n+3} + F_{3n}\right)  - F_{n+1}^3 - F_n^3 + \frac{1}{2},
    \end{equation}
    and
    \begin{equation}\label{e3}
        \sum_{k=1}^n (-1)^k F_{k}^3 \ = \
         \frac{1}{4} \left( (-1)^n F_{3n+3} + (-1)^{n+1} F_{3n}\right) -  (-1)^{n} F_{n+1}^3 - (-1)^{n+1} F_{n}^3 + \frac{1}{2}.
    \end{equation}
\end{thm}

\begin{proof}[Proof of \eqref{eq12}]
It follows from \eqref{e2} and \eqref{e3} that
    \begin{align*}
\sum_{k=1}^{2m-1} (-1)^{k-1} F_k^3 &\ = \ - \sum_{k=1}^{2m-1} (-1)^{k} F_k^3 \ = \ \frac{F_{6m}}{4} - \frac{F_{6m-3}}{4} - F^3_{2m} + F^3_{2m-1} - \frac{1}{2},\mbox{ and }\\
   \sum_{k=2}^{2m-2} F_k^3 &\ = \ -1 + \sum_{k=1}^{2m-2} F_k^3 \ = \ \frac{F_{6m-3}}{4} + \frac{F_{6m-6}}{4} - F_{2m-1}^3 - F_{2m-2}^3 - \frac{1}{2}.
    \end{align*}
Hence, the left-hand side of \eqref{eq12} becomes
    \begin{align*}
        T_m&\ :=\ \left( \frac{F_{6m}}{4} - \frac{F_{6m-3}}{4} - F_{2m}^3 + F_{2m-1}^3 - \frac{1}{2} \right) F_{2m-1}^3 + \\
        & \quad\quad\quad\quad\quad\quad\quad\quad\quad\left(\frac{F_{6m-3}}{4} + \frac{F_{6m-6}}{4} - F_{2m-1}^3 - F_{2m-2}^3 - \frac{1}{2} \right) F_{2m}^3 \\
        &\ = \  \frac{F_{6m} - F_{6m-3}}{4} F_{2m-1}^3 + F_{2m-1}^6 + \frac{F_{6m-3} + F_{6m-6}}{4} F_{2m}^3 - \\
        & \quad\quad\quad\quad\quad\quad\quad\quad\quad F_{2m-2}^3 F_{2m}^3- 2 F_{2m-1}^3 F_{2m}^3 - \frac{F_{2m-1}^3 + F_{2m}^3}{2}.
    \end{align*}
    Then the identity $T_m = (F_{2m-1}^3 - 1) (F_{2m}^3 - 1) / 2$ is equivalent to 
    \begin{align}\label{e4}
    {(F_{6m} - F_{6m-3})} F_{2m-1}^3 + 4 F_{2m-1}^6 & + {(F_{6m-3} + F_{6m-6})} F_{2m}^3\nonumber\\
    &- 4 F_{2m-2}^3 F_{2m}^3- {10 F_{2m-1}^3 F_{2m}^3} \ =\ 2.
    \end{align}
    Since $F_{3m} = 5F_m^3 + 3(-1)^mF_m$, we write 
    $$
        F_{6m} \ =\ 5F_{2m}^3 + 3F_{2m}, \mbox{ } F_{6m-3} \ =\ 5F_{2m-1}^3 - 3F_{2m-1}, \mbox{ and }F_{6m-6} \ =\ 5F_{2m-2}^3 + 3F_{2m-2}.
    $$
    Hence, 
    \begin{equation}\label{e5}
    (F_{6m} - F_{6m - 3}) F_{2m-1}^3 \ =\ 5F_{2m}^3 F_{2m-1}^3 + 3F_{2m} F_{2m-1}^3 - 5F_{2m-1}^6 + 3 F_{2m-1}^4,
    \end{equation}
    and
    \begin{equation}\label{e6}
    {(F_{6m-3} + F_{6m-6})} F_{2m}^3 = 5F_{2m-1}^3 F_{2m}^3 - 3F_{2m-1} F_{2m}^3 +5F_{2m-2}^3 F_{2m}^3 + 3F_{2m-2} F_{2m}^3.
    \end{equation}
    Using \eqref{e5} and \eqref{e6}, we can rewrite the left-hand side of \eqref{e4} as 
    $$
         S_m\ :=\ 3 F_{2m} F_{2m-1}^3 - F_{2m-1}^6 + 3 F_{2m-1}^4 - 3 F_{2m-1} F_{2m}^3 + F_{2m-2}^3 F_{2m}^3 + 3 F_{2m-2} F_{2m}^3.
    $$
    By Cassini's identity,
    $$
    F_{2m-2} F_{2m} \ = \ (-1)^{2m-1} + F_{2m-1}^2 \ = \ -1 + F_{2m-1}^2,
    $$
    which gives
    $$
    F_{2m-2}^3 F_{2m}^3 \ = \ \left( -1 + F_{2m-1}^2 \right)^3 \ = \ -1 + 3 F_{2m-1}^2 - 3 F_{2m-1}^4 + F_{2m-1}^6,
    $$
    and 
    $$
    3F_{2m-2}F_{2m}^3 \ = \ 3F_{2m}^2 \left( -1 + F_{2m-1}^2 \right) \ = \ - 3F_{2m}^2 + 3 F_{2m}^2 F_{2m-1}^2.
    $$
    Therefore,
    \begin{align*}
        S_m &\ = \ 3F_{2m} F_{2m-1}^3 - 3 F_{2m-1} F_{2m}^3 - 1 + 3 F_{2m-1}^2 - 3 F_{2m}^2 + 3 F_{2m}^2 F_{2m-1}^2 \\
        &\ = \ 3 F_{2m}F_{2m-1} \left(F_{2m-1}^2 - F_{2m}^2 + F_{2m} F_{2m-1} \right) + 3 F_{2m-1}^2 - 3 F_{2m}^2 - 1 \\
        &\ = \ 3 F_{2m} F_{2m-1} \left(F_{2m-1}^2 - F_{2m}^2 + F_{2m} ( F_{2m} - F_{2m-2} ) \right) + 3 F_{2m-1}^2 - 3 F_{2m}^2 - 1 \\
        &\ = \ 3 F_{2m} F_{2m-1} \left(F_{2m-1}^2 - F_{2m} F_{2m-2} \right) + 3 F_{2m-1}^2 - 3 F_{2m}^2 - 1 \\
        &\ = \  3 F_{2m} F_{2m-1} + 3 F_{2m-1}^2 - 3 F_{2m}^2 - 1\\
        &\ =\ 3((F_{2m+1}-F_{2m-1})F_{2m-1} + F^2_{2m-1} - F^2_{2m}) - 1\\
        &\ =\ 3(F_{2m+1}F_{2m-1} - F^2_{2m}) - 1\ =\ 2,
    \end{align*}
as desired. 
\end{proof}

\begin{proof}[Proof of \eqref{eq13}] Using Theorem \ref{km}, we can rewrite
the left-hand side of \eqref{eq13} as
\begin{align*}T_m\ := \ 1 &+ \left(\frac{1}{4}(F_{6m+3}-F_{6m})-F^3_{2m+1} + F^3_{2m}-\frac{1}{2}\right)F^3_{2m}\\
&+ \left(\frac{1}{4}(F_{6m}+F_{6m-3})-F^3_{2m} - F_{2m-1}^3-\frac{1}{2}\right)F_{2m+1}^3.\end{align*}
Hence, $T_m = (F_{2m}^3-1)(F_{2m+1}^3-1)/2$ is equivalent to
\begin{equation}\label{eq14}4F_{2m}^6-10F_{2m}^3F_{2m+1}^3-4F_{2m-1}^3F_{2m+1}^3 + F_{2m}^3(F_{6m+3}-F_{6m}) + F_{2m+1}^3(F_{6m}+F_{6m-3}) = -2.\end{equation}
We now use the identity $F_{3m} = 5F_m^3 + 3(-1)^mF_m$  to write
$$
        F_{6m-3} \ =\ 5F_{2m-1}^3 - 3F_{2m-1}, \mbox{ } F_{6m} \ =\ 5F_{2m}^3 + 3F_{2m}, \mbox{ and }F_{6m+3} \ =\ 5F_{2m+1}^3 - 3F_{2m+1}.
$$
Hence, \eqref{eq14} becomes
\begin{equation}\label{eq15}-F_{2m}^6-3F_{2m}^4-3F_{2m+1}F_{2m}^3 + 3F_{2m+1}^3F_{2m} + F_{2m-1}^3F_{2m+1}^3 - 3F_{2m+1}^3F_{2m-1} \ =\ -2.\end{equation}

By Cassini's identity, 
\begin{equation}\label{eq16}F_{2m-1}F_{2m+1}\ =\ F_{2m}^2 + 1,\end{equation}
so
$$F_{2m+1}^3F^3_{2m-1} \ =\ F^6_{2m} + 3F_{2m}^4+3F_{2m}^2 + 1,$$
and 
\begin{equation}\label{eq17}3F_{2m+1}^3F_{2m-1}\ =\ 3F_{2m+1}^2(F_{2m}^2+1)\ =\ 3F_{2m}^2F_{2m+1}^2 + 3F_{2m+1}^2.\end{equation}
Let $S_m$ be the left side of \eqref{eq15}. By \eqref{eq16} and \eqref{eq17}, we have
\begin{align*}
S_m&\ =\ -3F_{2m+1}F^3_{2m} + 3F^3_{2m+1}F_{2m} + 3F^2_{2m} + 1 - 3F^2_{2m}F^2_{2m+1} - 3F^2_{2m+1}\\
&\ =\ 3F_{2m}F_{2m+1}(F^2_{2m+1}-F^2_{2m} - F_{2m}F_{2m+1}) + 3F^2_{2m} - 3F^2_{2m+1} + 1\\
&\ =\ 3F_{2m}F_{2m+1}(F^2_{2m+1}-F^2_{2m} - (F_{2m+1}-F_{2m-1})F_{2m+1}) + 3F^2_{2m} - 3F^2_{2m+1} + 1\\
&\ =\ 3F_{2m}F_{2m+1}(-F^2_{2m} + F_{2m-1}F_{2m+1}) + 3F^2_{2m} - 3F^2_{2m+1} + 1\\
&\ =\ 3F_{2m}F_{2m+1} + 3F^2_{2m} - 3F^2_{2m+1} + 1\\
&\ =\ 3\left((F_{2m+1}-F_{2m-1})F_{2m+1} + F_{2m}^2 - F^2_{2m+1}\right) + 1\\
&\ =\ 3(F^2_{2m}-F_{2m-1}F_{2m+1}) + 1\ =\ -2,
\end{align*}
as desired. 
\end{proof}

\section{Another pair of Diophantine equations}\label{anotherpair}

Up until now, we have been working with Equations (\ref{eq1}) and (\ref{eq2}) given in Theorem \ref{o1}, which admit a unique nonnegative, integral solution. We now turn our attention to the question of the existence of a distinct set of equations that instead admits a unique positive, integral solution. Such a system can be found trivially by shifting the equations in Theorem \ref{o1} by $a+b$. 

\begin{prop}
For relatively prime $a, b\in \mathbb{N}$, exactly one of the following equations has a nonnegative, integral solution:
\begin{align}
ax + by &\ =\ \frac{(a-1)(b-1)}{2} + (a + b),\\
 ax + by + 1 &\ =\ \frac{(a-1)(b-1)}{2} + (a + b).
\end{align}
Furthermore, the solution is unique. 
\end{prop}

Our next theorem presents a system that is not a linear shift of the equations in Theorem \ref{o1}.

\begin{thm}\label{m3}
Let $a, b \in \mathbb{N}$ satisfy $(a, b) = 1$, $b \geq 2$, and $a$ is odd. Consider the two following equations:
    \begin{align}
        \label{eq18}ax + by \ = \ \frac{(a+1)b}{2} + 1, \\
        \label{eq19}ax + by \ = \ \frac{(a+1)b}{2} - 1.
    \end{align}
    Exactly one of the two equations has a positive, integral solution, and the solution is unique.
\end{thm}

The following lemma shall be used in due course. 

\begin{lem}\label{l1}
    For integers $n, x, y, a, b$ with $a, b$ positive and $(a, b) = 1$, we consider the equation $xa + yb = n$. If there is a solution $(x,y) = (r, s)\in \mathbb{Z}^2$ with $r < b$ and $s \leq 0$, then there are no solutions with $x, y$ both positive.
\end{lem}

\begin{proof}
    All integral solutions are of the form $(x, y) = (r + tb, s - ta)$ for some $t \in \mathbb{Z}$. To get $y > 0$, we must have $t < 0$, but that implies $x < 0$.
\end{proof}

\begin{proof}[Proof of Theorem \ref{m3}]
    Let $k = (a + 1) b /2$. Let $1 \leq r_1 \leq b - 1$ be chosen such that $ar_1 \equiv k+1 \mod b$ and $s_1 \coloneq (k+1 - ar_1)/b$. Let $1 \leq r_2 \leq b - 1$ be chosen such that $ar_2 \equiv k - 1 \mod b$ and $s_2 \coloneq (k - 1 - ar_2) / b$. Observe that 
    $$
    a(r_1 + r_2) \ = \ ar_1 + ar_2 \ \equiv \ 2k \ = \ (a+1)b \ \equiv \ 0 \mod b.
    $$
    It follows from $(a, b) = 1$ that $b \mid (r_1 + r_2)$. Since $2 \leq r_1 + r_2 \leq 2b - 2$, we know that $r_1 + r_2 = b$. Hence,
    $$
    s_1 + s_2 \ = \ \frac{(k+1 - ar_1) + (k-1 - ar_2)}{b} \ = \ \frac{2k - a(r_1 + r_2)}{b} \ = \ \frac{(a+1)b - ab}{b} \ = \ 1.
    $$
    Thus, exactly one of $s_1, s_2$ is positive. By definition, $r_1 a + s_1 b = k + 1$ and $r_2 a + s_2 b = k - 1$. Suppose, without loss of generality, that \eqref{eq18} has a positive solution. We know that \eqref{eq19} has no positive solution due to Lemma \ref{l1}.

    It remains to show that \eqref{eq18} has at most one positive solution. Let $(x_1, y_1)$ and $(x_2, y_2)$ be two positive solutions of \eqref{eq18}. Observe that
    $$
    x_1, x_2\ \leq \ \frac{(a+1) b}{2} - 1,
    $$
    which gives
    $$
    x_1, x_2 \ \leq \ \frac{b}{2} + \frac{b}{2a} - \frac{1}{a} \ < \ b.
    $$
    Hence, $1 \leq x_1, x_2 \leq b - 1$, and so, $|x_1 - x_2| \leq b-2$. Furthermore, $x_1 a + y_1 b = x_2 a + y_2 b$ implies that
    $$
    (x_1 - x_2) a \ = \ - (y_1 - y_2) b.
    $$
    Since $(a, b) = 1$, $b$ divides $x_1 - x_2$, which, in combination with $|x_1 - x_2| \le b - 2$, implies that $x_1 = x_2$. As a result, $(x_1, y_1) = (x_2, y_2)$. Therefore, \eqref{eq18} has at most one positive solution.
\end{proof}

\section{Problems for further investigation}\label{problems}
A natural question would be, for every fixed $i\ge 4$, whether there is a formula (similar to the ones in \cite[Theorem 1.6]{C4} and Theorems \ref{m1} and \ref{m2}) for the unique nonnegative, integral solution $(x,y)$ of the pair
\begin{align}
    F_n^ix_n + F_{n+1}^iy_n &\ =\ \frac{(F_n^i-1)(F_{n+1}^i-1)}{2},\\
    1+ F_n^ix_n + F_{n+1}^iy_n &\ =\ \frac{(F_n^i-1)(F_{n+1}^i-1)}{2}.
\end{align}
We can follow the same process that we use to obtain Theorems \ref{m1} and \ref{m2}: collecting data and then looking for any pattern. The drawback is that when $i$ is big, collecting data for $\Gamma(F_n^i, F_{n+1}^i)$ becomes more difficult, which limits our ability to observe a pattern  (if any) for $\Gamma(F_n^i, F_{n+1}^i)$. For instance, in collecting data for $\Gamma(F_n^4, F_{n+1}^4)$, we face a memory error at $\Gamma(F_{12}^4, F_{13}^4)$. While there are some patterns in Table 3 that may hint at a possible formula for the solution $(x_n, y_n)$ such as
$$
y_4-x_3\ =\ 1, y_5 - x_4 \ =\ -1, y_7-x_6 \ =\ 1, y_8-x_7 \ =\ -1, y_{10}-x_9 \ =\ 1, y_{11}-x_{10}\ =\ -1,
$$
there is no clear periodicity for $\Gamma(F_n^4, F_{n+1}^4)$, at least from the first $10$ values in Table 3.

\begin{center}
\begin{tabular}{ |c|c|c|c|c|c| } 
 \hline
 $n$ & $F_n^4$ & $F_{n+1}^4$ & $x_n$ & $y_n$ & $\Gamma(F_n^4, F_{n+1}^4)$ \\
 \hline
2 & 1 & 16 & 0 & 0 & 1\\
3 & 16 & 81 & 2 & 7 & 2\\
4 & 81 & 625 & 285 & 3 & 1\\
5 & 625 & 4096 & 183 & 284 & 2\\
6 & 4096 & 28561 & 1286 & 1863 & 2\\
7 & 28561 & 194481 & 88473 & 1287 & 1\\
8 & 194481 & 1336336 & 60247 & 88472 & 2\\
9 & 1336336 & 9150625 & 412554 & 607919 & 2\\
10 & 9150625 & 62742241 & 28542389 & 412555 & 1\\
11 & 62742241 & 429981696 & 19385711 & 28542388 & 1\\
 \hline
\end{tabular}
\begin{center}Table 3. Data for $\Gamma(F_n^4, F_{n+1}^4)$.\end{center}
\end{center}
Another unexpected feature of Table 3 is the non-monotonicity of $x_n$ and $y_n$, unlike what we have in Tables 1 and 2.

A more general problem is to find a formula for the solution of $\Gamma(F_n^i, F_{n+1}^j)$ for arbitrary $i$ and $j$. As an example, we collect the data for $\Gamma(F_n^2, F_{n+1}^3)$:

\begin{center}
\begin{tabular}{ |c|c|c|c|c|c| } 
 \hline
 $n$ & $F_n^2$ & $F_{n+1}^3$ & $x_n$ & $y_n$ & $\Gamma(F_n^2, F_{n+1}^3)$ \\
 \hline
2 & 1 & 8 & 0 & 0 & 1\\
3 & 4 & 27 & 3 & 1 & 1\\
4 & 9 & 125 & 55 & 0 & 2\\
5 & 25 & 512 & 20 & 11 & 1\\
6 & 64 & 2197 & 51 & 30 & 1\\
7 & 169 & 9261 & 4493 & 2 & 2\\
8 & 441 & 39304 & 356 & 216 & 1\\
9 & 1156 & 166375 & 935 & 571 & 1\\
10 & 3025 & 704969 & 350037 & 10 & 2\\
 \hline
\end{tabular}
\begin{center}Table 4. Data for $\Gamma(F_n^2, F_{n+1}^3)$.\end{center}
\end{center}
Table 4 suggests that $\Gamma(F_n^2, F_{n+1}^3)$ follows the same pattern as $\Gamma(F_n^2, F_n^2)$. 

Finally, all of the above questions can be asked for the pair in Theorem \ref{m3}.


 \ \\

\end{document}